\documentclass[12pt,reqno]{amsart}

\usepackage{mathrsfs}
\usepackage{amsgen}
\usepackage{amscd}
\usepackage{amsmath}
\usepackage{latexsym}
\usepackage{amsfonts}
\usepackage{amssymb}
\usepackage{amsthm}
\usepackage{graphicx}
\usepackage{color}
\usepackage{amsthm,amssymb}
\usepackage{graphicx}
\usepackage{enumerate}
\usepackage{amssymb,amsmath,graphicx,amsfonts,euscript}
\usepackage{color}
\usepackage{mathrsfs}
\usepackage{empheq}
\usepackage{cases}
\usepackage{cite}
\usepackage[colorlinks=true,backref]{hyperref}

\usepackage[colorlinks=true,backref]{hyperref}
\hypersetup{urlcolor=blue, citecolor=blue, linkcolor=black}

\usepackage[margin=3cm, a4paper]{geometry}

\newtheorem{thm}{Theorem}[section] 
\newtheorem{cor}[thm]{Corollary}
\newtheorem{lem}[thm]{Lemma}
 \newtheorem{prop}[thm]{Proposition}
\theoremstyle{remark} 
\newtheorem{rem}[thm]{Remark}

\newcommand{\EQ}[1]{\begin{align*}\begin{split} #1 \end{split}\end{align*}}
\newcommand{\EQn}[1]{\begin{align}\begin{split} #1 \end{split}\end{align}}
\newcommand{\EQnn}[1]{\begin{align} #1 \end{align}}

\newcommand{\Del}[1]{}

\def\norm#1{\left\|#1\right\|}
\def\normo#1{\|#1\|}
\def\normb#1{\big\|#1\big\|}

\def\abs#1{|#1|}
\def\aabs#1{\left|#1\right|}
\def\absb#1{\big|#1\big|}

\def\brk#1{\left(#1\right)}
\def\brko#1{(#1)}
\def\brkb#1{\big(#1\big)}

\def\fbrk#1{\left\lbrace#1\right\rbrace} 
\def\fbrkb#1{\big\lbrace#1\big\rbrace} 
\def\jb#1{\langle#1\rangle}

\def\wt#1{\widetilde{#1}}
\def\wh#1{\widehat{#1}}
\def\wb#1{\overline{#1}}

\def\pdt{\partial_{t}}

\newcommand{\ra}{{\rightarrow}}

\def\loe{\leqslant}
\def\goe{\geqslant}
\def\lsm{\lesssim}
\def\gsm{\gtrsim}

\newcommand{\N}{{\mathbb N}}

\newcommand{\R}{{\mathbb R}}
\newcommand{\C}{{\mathbb C}}
\newcommand{\Z}{{\mathbb Z}}

\newcommand{\F}{{\mathcal{F}}}

\newcommand{\E}{{\mathcal{E}}}

\def\dx{\text{\ d} x}
\def\ds{\text{\ d} s}
\def\dt{\text{\ d} t}

\def\ep{\varepsilon}
\def\al{\alpha}

\def\ph{\varphi}

\def\de{\delta}
\def\De{\Delta}

\def\la{\lambda}
\def\ga{\gamma}

\newcommand{\I}{\infty}
\def\rev#1{\frac{1}{#1}}
\def\half#1{\frac{#1}{2}}

\def\vl{v_{\mathrm{li}}}
\def\vnl{v_{\mathrm{nl}}}

\numberwithin{equation}{section} 

\allowdisplaybreaks[4]
\begin{document}
\title[3D NLS]{Global well-posedness and scattering of 3D defocusing, cubic Schr\"odinger equation}

\author{Jia Shen}
\address{(J. Shen) Center for Applied Mathematics\\
	Tianjin University\\
	Tianjin 300072, China}
\email{shenjia@tju.edu.cn}

\author{Yifei Wu}
\address{(Y.Wu) Center for Applied Mathematics\\
Tianjin University\\
Tianjin 300072, China}
\email{yerfmath@gmail.com}
\thanks{}

\date{}

\begin{abstract}
In this paper, we study the global well-posedness  and scattering of 3D defocusing,  cubic Schr\"odinger equation. Recently,
Dodson \cite{Dod20NLS} studied the global well-posedness  in a critical Sobolev space $\dot{W}^{11/7,7/6}$.  In this paper, we aim to show that if the initial data belongs to $\dot H^\frac12$ to guarantee the local existence, then some extra weak space which is supercritical, is sufficient to prove the global well-posedness. More precisely, we prove that if the initial data belongs to $\dot{H}^{1/2}\cap \dot{W}^{s,1}$ for $12/13<s\loe 1$, then the corresponding solution exists globally and scatters. %\red{Furthermore, in the case when $s<1$, we do not have any $\nabla$-estimate for linear solution, which is the main obstruction in our argument.}

\end{abstract}

\maketitle

\section{Introduction}
We study the non-linear Schr\"odinger equation (NLS)
\begin{equation}
\label{eq:nls}
\left\{ \aligned
&i\pdt u + \De u =  \mu|u|^p u, \\ 
& u(0,x) = u_0(x),
\endaligned
\right.
\end{equation}
where $u(t,x):\R\times\R^d\ra \C$, $\mu=\pm1$ and $p>0$. $\mu=1$ is \textit{defocusing} case, and $\mu=-1$ is \textit{focusing} case. The equation \eqref{eq:nls} has conserved mass
\EQn{
	M(u(t)) :=\int_{\R^d} \abs{u(t,x)}^2 \dx=M(u_0),
}
and energy
\EQn{
	E(u(t)) :=\int_{\R^d} \half 1\abs{\nabla u(t,x)}^2 \dx + \mu \int_{\R^d} \rev{p+2} \abs{u(t,x)}^{p+2} \dx=E(u_0).
}
The equation \eqref{eq:nls} also has rescaled solution
\EQn{\label{scaling}
	u_\la(t,x) = \la^{2/p} u(t/\la^2,x/\la).
}
The scaling  leaves  $\dot{H}^{s_{c}}, s_c=\frac d2-\frac 2p$ norm invariant, that is,
$
\|u(0)\|_{\dot H_x^{s_{c}}}=\|u_{\lambda}(0)\|_{\dot H_x^{s_{c}}}.
$
In this sense, it has critical Sobolev space $\dot{H}_x^{s_c}$. Therefore, according to the conservation law, the equation is called mass critical when $p=\frac 4d$, and energy critical when $p=\frac4{d-2}$.

The local well-posedness for \eqref{eq:nls} is well-understood, which can be found in Cazenave and Weissler \cite{CW1}. 
Moreover, there are extensive studies on the large data global well-posedness and scattering for \eqref{eq:nls}. Let us first mention the results in energy space when $\mu=1$. In energy subcritical and mass supercritical case, Ginibre and Velo \cite{GV85JMPA} proved the large data global well-posedness and scattering in energy space for $d\goe 3$. Bourgain \cite{Bou99JAMS} introduced the induction on energy method to study the 3D quintic energy critical NLS and proved global well-posedness and scattering in $H^1(\R^3)$ for radial data.  Nakanishi \cite{Nak99JFA} used this method and introduced a new kind of modified Morawetz estimate to obtain the energy scattering for $p>4/d$ when $d=1$ and $d=2$. Bourgain's result was extended to the non-radial case by Colliander, Keel, Staffilani, Takaoka and Tao \cite{Iteam08Annals}, for which a key ingredient is the interaction Morawetz estimate introduced in \cite{Iteam04CPAM}. Energy critical case in higher dimensions was solved by \cite{RV07AJM,Vis07Duke}. Secondly, focusing equations has special solutions constructed from ground state, which does not scatter. Kenig and Merle \cite{KM06Invent} introduced the concentration compactness method to study the radial focusing energy critical NLS in $\dot{H}^1$ when $d=3,4,5$. They proved the ``ground state theorem", namely below the energy of ground state, the solution with positive virial is global well-posed and scatters, and the solution with negative virial blows up. Killip and Visan \cite{KV10AJM} obtained the scattering below the ground state for non-radial energy critical NLS in dimensions five and higher, and dimension $d=4$ case was solved by Dodson \cite{Dod19ASENS}.  In mass critical case, authors further studied the global well-posedness and scattering purely in $L^2$. In radial case, it was solved by Killip, Tao and Visan \cite{KTV10JEMS} for $d=2$, and by Tao, Visan, Zhang \cite{TVZ07Duke}, and Killip, Visan and Zhang \cite{KVZ08APDE} for $d\goe 3$. The defocusing non-radial case for all dimensions were proved by Dodson \cite{Dod12JAMS,Dod16Duke,Dod16AJM}. Furthermore, Dodson \cite{Dod15Adv}  studied the focusing case with initial data below the mass of the ground state in all dimensions.

Next, we review some of the results for the energy subcritical and mass supercritical case. As the typical model, we consider the 3D cubic NLS as follows, 
\begin{equation}
\label{eq:nls-cubic}
\left\{ \aligned
&i\pdt u + \De u =  \mu|u|^2 u, \\ 
& u(0,x) = u_0(x),
\endaligned
\right.
\end{equation}
In the defocusing case, as described above, global and scattering in energy space was proved by Ginibre and Velo in \cite{GV85JMPA}. In the focusing case, global well-posedness and scattering in energy space  below the ground state was proved by Duyckaerts, Holmer and Roudenko in \cite{HR08CMP,DHR08MRL}. For negative virial  data,  blowing-up solution in energy space was established  in \cite{Glassey77, OgTs91, HoRo2, DuWuZhang}.  Below the energy space, Bourgain \cite{Bou98JAM} obtained global well-posedness and scattering in $H^s$ with $s>11/13$ in the defocusing case. The result was further extended to $s>5/6$ in \cite{Iteam02MRL}, $s>4/5$ in \cite{Iteam04CPAM} and $s>5/7$ in \cite{Su12MRL}. With radial assumption, Dodson \cite{Dod18CJM} obtained the result almost in critical space $s>1/2$. In critical space, Kenig and Merle \cite{KM10TranAMS} proved global well-posedness and scattering for defocusing case with assumption that the solution is uniformly bounded in critical space. More recently, Dodson \cite{Dod20NLS} studied the global well-posedness for defocusing equation in a critical Sobolev space $\dot{W}^{11/7,7/6}$.

In this paper, we intend to study the global well-posedness and scattering for defocusing 3D cubic NLS in critical $\dot{H}^{1/2}$ space intersected with a supercritical space. The main result is
\begin{thm}\label{thm:main}
	Let $\mu=1$ and $1\goe s> 12/13$. Suppose that $u_0\in \dot{H}_x^{1/2}\cap \dot{W}_x^{s,1}$, then the solution $u$ of equation $\eqref{eq:nls-cubic}$ exists globally and scatters.
\end{thm}
\begin{rem}\label{rem:mainresut}
	The followings are some remarks related to the theorem.
\begin{enumerate}
		\item 
		Our main purpose is to show that if the initial data belongs to $\dot H_x^{1/2}$ to guarantee the local existence, then some extra weak space, likes $\dot W^{s,1}$ which is supercritical space under the scaling, is sufficient to prove the global well-posedness.
		\item 
		We mainly concern the rough data case when the derivative index $s<1$. In this case, the main obstruction is the lack of any $\nabla$-estimate for the linear solution. Then, the two key ingredients are as follows. First, we invoke the bilinear Strichartz estimate suitably to lower down the derivative of the linear solution, but it costs the increase of the energy bound. To overcome the difficulty, we also derive a global space-time estimate $L_t^q L_x^r$ with $r>6$, see \eqref{eq:globalbound-w-l2infty}. These tools allow us to control the growth of the energy under the rough data assumption.
		\item 
		The scattering statement is also  included. Indeed, by the uniform boundedness in $\dot H_x^{1/2}$ obtained below, it is an immediate consequence of the result of Kenig and Merle \cite{KM10TranAMS}. 
		\item 
		Lastly, we remark that the lower bound $12/13$ for the regularity $s$ of the space $\dot W^{s,1}$ is not optimal, which can be improved by more delicate analysis, see \eqref{eq:globalbound-w-full} below. In particular, we can further lower down the required regularity by $I$-method in \cite{Iteam04CPAM}. However, in this paper we are not going to pursue this optimality. We conjecture that if $u_0\in \dot{H}^{1/2}\cap L^1$, then the corresponding solution of equation $\eqref{eq:nls-cubic}$ with $\mu=1$ exists globally and scatters. 
\end{enumerate}
\end{rem}

The method in  \cite{Dod20NLS} was based on the fact that if the initial data belongs to $W^{s,p}$ for $p<2$, then the linear flow becomes more regular by the dispersive equation when the time is away from zero. The similar observation was made in the earlier paper  \cite{BDSW}, in which Beceanu, Deng, Soffer, and the second author showed that if the initial data is radial and compactly supported, then the linear flow becomes smoother when the time is away from zero. Due to this, they proved the global well-posedness in mass-subcritical case when the initial data $u_0\in   \dot H_x^{s_c}$, radial and compactly supported. This fact also plays a crucial role in our argument, see Lemma \ref{lem:globalbound-v} below for the precise statement.   

Let us explain the main idea of the proof in this paper. We use the argument in \cite{BDSW}, in which a strategy of ``time-cutoff'' equation's decomposition  was given. More precisely, we split the solution $u=v+w$ with $v$ satisfying 
$$
i\partial_tv+ \De v =  \chi_{\le1}(t)|v|^2 v.
$$  
Then we can prove that $v$ keeps the properties as the linear flow.  Furthermore, $w$ satisfies 
$$
i\partial_tw+ \De w =  \chi_{\gtrsim1}(t)|v|^2 v+O(wu^2).
$$
As the heart of the whole analysis, benefiting from that $v$ is more regular in $t\gtrsim 1$, we can transfer the additional regularity from $v$ to $w$ by bootstrap. However, the  implementation of  such ideas is quite challenging because of the rough data. 

The first main obstruction is to improve the regularity of the solution $w$ when it transits from $t=0$ to $t=1$, we choose various appropriate function spaces to bootstrap. This is analogue to the works that studied the NLS at critical regularity, in which the Gronwall's inequalities were employed suitably to gain the regularity of the solution, see, e.g., \cite{KTV10JEMS,KVZ08APDE}.
To do this, the analysis is subtle and the estimates should be critical.   More precisely, to establish the $H^1$-estimate of $w$ in local time for $H^\frac12$ data, there is a half of derivative loss, and we use the bi-linear Strichartz estimate to save the regularity. In particular, a multi-scale bi-linear Strichartz estimate established very recently by Candy \cite{Can19MathAnn} shall be applied to overcome the difficulties. 

More precisely, the  bi-linear Strichartz estimate in the form of
\EQn{\label{eq:bilinear}
\norm{[e^{it\De}\phi] [e^{\pm it\De}\psi]}_{L_t^q L_x^r(\R\times\R^d)}
}
with particular frequency restrictions on $\phi,\psi$ plays a crucial role in this paper. $L_{t,x}^2$ bi-linear estimate for Schr\"odinger equation was first introduced by Bourgain \cite{Bou98IMRN} in 2D case, and was further extended in \cite{Iteam08Annals} and \cite{Vis07Duke}. This kind of estimates was widely used in the theory of critical non-linear Schr\"odinger equations, see for examples \cite{Iteam08Annals,Vis07Duke,KTV10JEMS,KVZ08APDE,Dod12JAMS,Dod15Adv,Dod16AJM,Dod16Duke}. For $q,r\loe 2$, \eqref{eq:bilinear} has close relation with bi-linear restriction estimate for paraboloid, which was extensively studied before. Wolff proved bi-linear restriction estimate for cone in  \cite{Wol01Annals} with $q=r>\frac{d+3}{d+1}$, and then Tao \cite{Tao03GAFA} extended the result to paraboloid case. Moreover, bi-linear restriction estimate is an important technique towards the linear restriction problems, see \cite{TVV98JAMS}. We refer the reader to Mattila's book \cite{Mat15book} for more references of bi-linear restriction theory. In this paper, we are going to use the multi-scale version of bi-linear restriction estimate by Candy \cite{Can19MathAnn}. To be more precise, in the proof of local result, in order to close the contraction mapping, we combine $L_t^{1+}L_x^2$ and the classical $L_{t,x}^2$ estimates to obtain additional regularity for the low frequency summation.

The second main ingredient of this paper is that the energy of $w$ part is almost conserved for any long time. This is achieved by controlling the energy increment of the perturbed solution $w$. More precisely, the energy increment contains at least one $v$ that has subcritical estimate, which allows us to control the remaining $w$ parts with supercritical norms. Since $v$ is lack of $\dot W^{1,\I}$-estimate ( especially when we assume $s<1$), we need to use the bi-linear Strichartz estimate to lower down the regularity of $v$. However, as described in Remark \ref{rem:mainresut} (b), in the view of \eqref{eq:globalbound-w-full} and \eqref{eq:globalbound-w-full2}, the use of bi-linear Strichartz estimate will increase the energy bound dramatically for long time. To close the bootstrap argument, we interpolate the bi-linear estimate and the interaction Morawetz estimate to control the energy increment properly, see \eqref{esti:energybound-mainterm-transverse-simple-1} below.

\textbf{Organization of the paper.}  In Section 2, we give some preliminaries. This includes
notations and some useful lemmas. In Section 3, we give the subcritical estimate of $v$. In Section 4, we give the energy estimate of $w$, and finish the proof of Theorem \ref{thm:main}.

\section{Preliminary}
\subsection{Notations}
The followings are some notations.

$\bullet$ $\hat f$ or $\F f$ denotes the Fourier transform of $f$.

$\bullet$ For any $a\in\R$, $a\pm:=a\pm\ep$ for some small $\ep>0$. 

$\bullet$ $C>0$ denotes some constant, and $C(a)>0$ denotes some constant depending on coefficient $a$.

$\bullet$ If $f\loe C g$, we write $f\lsm g$. If $f\loe C g$ and $g\loe C f$, we write $f\sim g$. Suppose further that $C=C(a)$ depends on $a$, then we write $f\lsm_a g$ and $f\sim_a g$, respectively. If $f\loe 2^{-5}g$, we denote $f\ll g$. The notation $f\gg g$ is defined similarly.

$\bullet$ $\abs{\nabla} := \F^{-1}|\xi|\F $ and $\abs{\nabla}^s:=\F^{-1}|\xi|^s\F $.  

$\bullet$ Take a cut-off function $\chi\in C_{0}^{\infty}(0,\infty)$ such that $\chi(r)=1$ if $r\loe1$ and $\chi(r)=0$ if $r>2$. For $N\in 2^\Z$, let $\chi_N(r) = \chi(N^{-1}r)$ and $\phi_N(r) =\chi_N(r)-\chi_{N/2}(r)$. We define the Littlewood-Paley dyadic operator $f_{\loe N}=P_{\loe N} f := \mathcal{F}^{-1}\brko{ \chi_N(|\xi|) \hat{f}(\xi)}$ and $f_{N}= P_N f := \mathcal{F}^{-1}\brko{ \phi_N(|\xi|) \hat{f}(\xi)}$. We also define that $f_{\goe N}=P_{\goe N} f := f- P_{\loe N} f$, $f_{\ll N}=P_{\ll N} f :=  P_{\loe 2^{-5} N} f$, $f_{\gsm N}=P_{\gsm N}f:=P_{\goe 2^{-5} N}f$, $f_{\lsm N}=P_{\lsm N}f:=P_{\loe 2^{5} N}f$, and $f_{\sim N}=P_{\sim N}:= P_{\loe 2^5 N} f - P_{\loe 2^{-5}N} f$.
%and for $N_1,N_2\in2^\Z$,
%$f_{N_1\loe\cdot\loe N_2} = P_{N_1\loe\cdot\loe N_2} f := P_{\loe N_2} f - P_{\loe N_1} f $.

$\bullet$ Let $\mathcal S(\R^d)$ be the Schwartz space and $\mathcal S'(\R^d)$ be the tempered distribution space. $L^p(\R^d)$ denotes the usual Lebesgue space. We define the homogeneous Sobolev and Besov spaces
\EQ{
&\dot W^{s,p}(\R^d) := \fbrkb{f \in\mathcal S'(\R^d): \sum_{N\in2^\Z} N^sP_Nf\text{ converges in }\mathcal S'\text{ to a $L^p$-function}}, \\
&\quad\quad\quad\dot B^s_{p,q}(\R^d) := \fbrkb{f \in\mathcal S'(\R^d): \norm{f}_{\dot B^s_{p,q}} := \normb{N^s\norm{f_N}_{L_x^p(\R^d)}}_{l_N^q(2^\Z)} }.
}
Moreover, the Sobolev norm is defined by $\norm{f}_{\dot{W}^{s,p}(\R^d)}:= \norm{\sum_{N\in2^\Z} N^sP_Nf}_{L^p(\R^d)}$. We know from the Littlewood-Paley theory that for $1<p<\I$, $\norm{f}_{\dot{W}^{s,p}(\R^d)}\sim \norm{|\nabla|^sf}_{L^p(\R^d)}$. We denote that $\dot{H}^s(\R^d):=\dot{W}^{s,2}(\R^d)$. We also define $\jb{\cdot,\cdot}$ as real $L^2$ inner product:
\EQ{
\jb{f,g} = \Re\int f(x)\wb{g}(x)\dx.
}

$\bullet$ For any time interval $I\subset\R$, we denote $L_t^q \dot{W}_x^{s,r}(I):=L_t^q \dot{W}_x^{s,r}(I\times\R^3)$ for short. 

$\bullet$ For any $0\loe\gamma\loe1$, we call that the exponent pair $(q,r)\in\R^2$ is $\dot H^\ga$-$admissible$, if $\frac{2}{q}+\frac{3}{r}=\half 3-\ga$, $2\loe q\loe\I$, and $2\loe r<\I$. If $\ga=0$, we say that $(q,r)$ is $L^2$-$admissible$.

$\bullet$ For functions $f,g,h$, we use the notation
\EQ{
O(fg):=c_1 fg + c_2f\wb{g} + c_3 \wb{f}g +c_4 \wb{f}\wb{g}, 
}
for some $c_i\in \C$, $i=1,2,3,4$. We also set
$O(fgh):=O((fg)h)$, and $O(f+g):=O(f) +O(g)$.

\subsection{Useful lemmas}
In this subsection, we gather some useful results.
\begin{lem}[Strichartz estimate,\cite{KT98AJM}]\label{lem:strichartz}
	Suppose that $(q,r)$ and $(\wt{q},\wt{r})$ are $L^2$-admissible.
	Then, we have
	\EQn{\label{eq:strichartz-1}
		\norm{ e^{it\De}\ph}_{L_t^qL_x^r(\R)} \lsm \norm{\ph}_{L_x^2},
	}
and
	\EQn{\label{eq:strichartz-2}
		\normb{\int_0^t e^{i(t-s)\De} F(s)\ds}_{L_t^qL_x^r(\R)} \lsm \norm{F}_{L_t^{\wt{q}'} L_x^{\wt{r}'}(\R)}.
	}
Furthermore, by   the Bernstein inequality and \eqref{eq:strichartz-1}, for any $0<\ep<\frac{1}{10}$,
\EQn{\label{eq:strichartz-3}
\brkb{\sum_{N\in2^\Z} N\norm{P_N e^{it\De} \ph}_{L_t^2L_x^\I(\R)}^2 + N^{4\ep-1} \norm{P_N e^{it\De} \ph}_{L_t^{1/\ep}L_x^\I(\R)}^2}^{\half 1}\lsm \norm{\ph}_{\dot H_x^{1}}.
}
\end{lem}

\begin{lem}[Schur's test]\label{lem:schurtest}
	For any $a>0$, sequences $\fbrk{a_N}$,  $\fbrk{b_N}\in l_{N\in2^\Z}^2$, we have
	\EQn{\label{eq:schurtest}
	\sum_{N_1,N\in2^\Z:N_1\loe N} \brkb{\frac{N_1}{N}}^a a_N b_{N_1} \lsm \norm{a_N}_{l_N^2} \norm{b_N}_{l_N^2}.
	}
\end{lem}
In this paper, we need the following multi-scale bi-linear Strichartz estimate for Schr\"odinger equation, which is a particular case of Theorem 1.2 in \cite{Can19MathAnn}: 
\begin{lem}\label{lem:bilinearstrichartz-origin}
	Let $1\loe q,r \loe 2$, $\rev q + \frac{2}{r}<2$, and suppose that $M,N\in2^\Z$ satisfy $M\ll N$. Then for any $\phi,\psi\in L_x^2(\R^3)$,
	\EQn{
	\norm{[e^{it\De}P_N\phi ][e^{\pm it\De}P_M \psi]}_{L_t^qL_x^r(\R)} \lsm &  \frac{M^{4-\frac{4}{r}-\frac{2}{q}}}{N^{1-\rev r}}\norm{P_N\phi}_{L_x^2}\norm{P_M\psi}_{L_x^2}.
	}
\end{lem}
Using the same argument as in \cite{Vis07Duke}, we can transfer the bi-linear estimate in Lemma \ref{lem:bilinearstrichartz-origin} from linear solutions into general functions:
\begin{lem}\label{lem:bilinearstrichartz}
	Let $I\subset\R$, $1\loe q,r \loe 2$, $\rev q + \frac{2}{r}<2$, and suppose that $M,N\in2^\Z$ satisfy $M\ll N$. Let $(\wt q,\wt r)$ be $L^2$-admissible with $\wt q'<q$. Moreover, for any $t\in I$,  $\wh{u}(t,\cdot)$ is supported on $\fbrk{\xi:\abs{\xi}\sim N}$, and $\wh{v}(t,\cdot)$ is supported on $\fbrk{\xi:\abs{\xi}\sim M}$. Then,
		\EQn{\label{eq:bilinearstrichartz}
		\norm{O\brk{uv}}_{L_t^qL_x^r(I)} \lsm &  \frac{M^{4-\frac{4}{r}-\frac{2}{q}}}{N^{1-\rev r}}\norm{u}_{S^*(I)} \norm{v}_{S^*(I)},
	}
	where $a\in I$, and
	\EQn{
		\norm{u}_{S^*(I)}:=\norm{u(a)}_{L_x^2}+\norm{(i\pdt +\De)u}_{L_t^{\wt q'}L_x^{\wt r'}(I)}.
	}
\end{lem}

\begin{lem}[Inhomogeneous Strichartz,\cite{Fos05JHDE,Vil07TranAMS}]\label{lem:inhomogeneousstrichartz}
	Let $I\subset \R$. Suppose that $(q,r)$ satisfy
	\EQ{
	\text{$\frac{5}{4}< q<\I$, $2\loe r \loe 15$, }
	\frac{1}{q} + \frac{3}{r}<\frac{3}{2}\text{, and }\frac{2}{q}+\frac{3}{r}= 2.
	}
	Then, for $a\in I$, we have
	\EQn{\label{eq:inhomogeneousstrichartz}
	\normb{\int_a^t e^{i(t-s)\De} F(s,x)\ds}_{L_{t,x}^5(I)} \lsm \norm{F}_{L_t^{q'} L_x^{r'}(I)}.
	}
	Particularly, $(q,r)=(3/2,9/2)$ satisfies the above conditions.
\end{lem}

\begin{lem}[Interaction Morawetz inequality, \cite{Iteam04CPAM}]\label{lem:interaction}
	Let $u\in C\brk{[0,T]:H^{1/2}}$ be the solution of \eqref{eq:nls-cubic} with $\mu=1$ for some $T>0$. Then, we have
	\EQn{
	\int_{0}^T\int_{\R^3} |u(t,x)|^4\dx\dt \lsm \norm{u_0}_{L_x^2}^2\sup_{t\in[0,T]}\norm{u(t)}_{\dot{H}_x^{1/2}}^2.
	}
\end{lem}

\subsection{Rescaling}
From now on, we assume that $\mu=1$ in \eqref{eq:nls-cubic}, and for convenience we denote the initial data as $\wt u_0$. Fix $\wt u_0\in \dot H_x^{1/2} \cap \dot W_x^{s,1}$ with $\frac{12}{13}<s\loe1$. Then, for any $\de>0$, there exists $t_0=t_0(\de,\wt u_0)<1$ such that
\EQ{
\norm{e^{it\De}\wt u_0}_{L_{t,x}^5([0,3t_0])} \loe \de.
}
Let $\de=\de(\norm{\wt u_0}_{\dot H_x^{1/2}\cap \dot W_x^{s,1}})>0$ be a small constant determined later. By standard local theory, the equation \eqref{eq:nls-cubic} admits a unique solution $\wt u\in C([0,t_0];\dot H^{1/2})$ with initial data $\wt u_0$, satisfying
\EQ{
\norm{\wt u}_{L_t^\I \dot H_x^{\frac12}([0,3t_0])} \loe 2\norm{\wt u_0}_{\dot H_x^{\frac12}}\text{, and }\norm{\wt u}_{L_{t,x}^5([0,3t_0])} \loe 2\de.
}
Then, we make the scaling transform
\EQ{
u_0:= t_0^{\frac12}\wt u_0(t_0^{\frac12}x)\text{, and } u(t,x):=t_0^{\frac12}\wt u(t_0t,t_0^{\frac12}x).
}
Now, we have
\EQn{\label{eq:initialbound-u}
\norm{u_0}_{\dot H_x^{1/2}} := \norm{\wt u_0}_{\dot H_x^{1/2}} \text{, and } \norm{u_0}_{\dot W_x^{s,1}} := t_0^{\frac s2-1}\norm{\wt u_0}_{\dot W_x^{s,1}}.
}
Moreover, the local solution $u$ of \eqref{eq:nls-cubic} is defined on $[0,3]$ with
\EQn{\label{eq:localestimate-u}
\norm{u}_{L_t^\I \dot H_x^{\frac12}([0,3])} \loe 2\norm{\wt u_0}_{\dot H_x^{\frac12}}\text{, and }\norm{u}_{L_{t,x}^5([0,3])} \loe 2\de.
}

\subsection{Decomposition of the solution}
For the above $\de$, we can find sufficiently large dyadic  $N_0=N_0(\de,t_0,\wt u_0)\in2^{\N}$ such that
\EQn{\label{eq:initialbound-v}
	\norm{P_{\goe N_0}u_0}_{\dot{H}_x^{1/2}} + \norm{P_{\goe N_0}u_0}_{\dot{W}_x^{s,1}} \loe t_0^{10}\de.
}
We define that $v_0=P_{\goe N_0}u_0$ and $w_0=u_0-v_0$. Then, we decompose the solution $u$ of \eqref{eq:nls-cubic} as $u(t,x) = v(t,x) + w(t,x)$, where $v$ and $w$ satisfy
\begin{equation}
\label{eq:small}
\left\{ \aligned
&i\pdt v + \De v =  \chi(t)|v|^2 v, \\ 
&v(0,x) =  v_0(x),
\endaligned
\right.
\end{equation}
and
\begin{equation}
\label{eq:energy}
\left\{ \aligned 
&i\pdt w + \De w =  |u|^2 u - \chi(t)|v|^2 v, \\ 
&w(0,x) =  w_0(x),
\endaligned
\right.
\end{equation}
respectively. In the following, we regard $\norm{\wt u_0}_{\dot H_x^{1/2}\cap \dot W_x^{s,1}}$ as a constant and omit its dependence for short. We consider the $L_x^2$-norm of $u_0$:
\EQ{
\norm{u_0}_{L_x^2}\loe &  \norm{P_{\loe 1}u_0}_{L_x^2} + \norm{P_{\goe 1}u_0}_{L_x^2}.
}
Note that
\EQ{
\norm{P_{\goe 1}u_0}_{L_x^2} \lsm \norm{P_{\goe 1}u_0}_{\dot H_x^{1/2}} \lsm 1.
}
By Bernstein's inequality, $s<\frac32$, $\dot W^{s,1}\hookrightarrow\dot B^s_{1,\I}$, and \eqref{eq:initialbound-u},
\EQ{
\norm{P_{\loe 1}u_0}_{L_x^2} 
\lsm  \sum_{N\loe1} \norm{P_Nu_0}_{L_x^2}
\lsm  \sum_{N\loe1} N^{\frac32} \norm{P_Nu_0}_{L_x^1} 
\lsm  \norm{u_0}_{\dot W^{s,1}} 
\lsm  t_0^{\frac s2-1}.
}
Therefore,
\EQn{\label{eq:initialbound-u-l2}
\norm{u_0}_{L_x^2}\lsm t_0^{\frac s2-1}.
}
By the above argument and \eqref{eq:initialbound-v},
\EQn{
\norm{v_0}_{L_x^2}\lsm t_0^{10}\de, \label{eq:initialbound-v-l2}
}
and
\EQn{\label{eq:initialbound-w}
	\norm{w_0}_{\dot{H}_x^1}\lsm N_0^{1/2} \norm{u_0}_{\dot{H}_x^{1/2}}\lsm N_0^{1/2}.
}
Using the equation, for any $0\loe\ga\loe1/2$ and $\dot H_x^{\half 1 -\ga}$-admissible $(q,r)$, we have
\EQn{\label{eq:localbound-u}
	\norm{\abs{\nabla}^\ga u}_{L_t^qL_x^r([0,3])} \lsm \brkb{\sum_{N\in2^\Z}N^{2\ga}\norm{ P_Nu}_{L_t^q L_x^r([0,3])}^2}^{1/2} \lsm \norm{u_0}_{\dot H_x^{1/2}} \lsm 1.
}

\section{Subcritical estimate of \textit{v}}
%\texorpdfstring{$v$}{\textit{v}}
%In this section, we consider the estimates on $v$. 
We define $\vl(t,x):=e^{it\De}v_0$ and
\EQ{
	\vnl(t,x):= -i\int_0^t e^{i(t-s)\De} \chi(s)|v(s,x)|^2 v(s,x) \ds.
}
By \eqref{eq:small}, we have $v=\vl+\vnl$. Using the Strichartz estimate in Lemma \ref{lem:strichartz},
\EQ{
	\norm{v}_{L_t^\I H_x^{1/2}\cap L_t^3L_x^9(\R)} \lsm \norm{v_0}_{H_x^{1/2}} + \norm{v}_{L_t^\I H_x^{1/2}(\R)} \norm{v}_{L_t^3L_x^9(\R)}^2.
}
Therefore, by \eqref{eq:initialbound-v}, \eqref{eq:initialbound-v-l2}, and the standard fixed point argument, we have that $v$ is global defined, and satisfies
\EQn{\label{eq:globalbound-v-subonehalfregularity}
	\norm{v}_{L_t^\I H_x^{1/2}\cap L_t^3L_x^9(\R)} \lsm t_0^{10}\de,
}
and thus by  \eqref{eq:strichartz-2} in Lemma \ref{lem:strichartz}, 
\EQn{\label{eq:globalbound-v-subonehalfregularity-NL}
	\norm{\vnl}_{L_t^\I H_x^{1/2}\cap L_t^3L_x^9(\R)} \lsm t_0^{30}\de^3.
}
Using the equation again, by the Bernstein inequality, the Strichartz estimate \eqref{eq:strichartz-2} in Lemma \ref{lem:strichartz}, and the Littlewood-Paley theory, we have the following $\dot H_x^{\half 1}$ estimate:
for any $0\loe\ga\loe1/2$, and $\dot H_x^{\half 1-\ga}$-admissible exponent pair $(q,r)$, we have
\EQn{\label{eq:globalbound-v-onehalfregularity}
	\norm{\abs{\nabla}^\ga v}_{L_t^qL_x^r(\R)} \lsm \brkb{\sum_{N\in2^\Z}N^{2\ga}\norm{ P_Nv}_{L_t^q L_x^r(\R)}^2}^{1/2} \lsm t_0^{10}\de.
}
Combining \eqref{eq:localbound-u}, \eqref{eq:localestimate-u} and \eqref{eq:globalbound-v-onehalfregularity},  we also have
\EQn{\label{eq:localbound-w-onehalf}
	\norm{\abs{\nabla}^\ga w}_{L_t^qL_x^r([0,3])} \lsm \brkb{\sum_{N\in2^\Z}N^{2\ga}\norm{ P_Nw}_{L_t^q L_x^r([0,3])}^2}^{1/2} \lsm \norm{u_0}_{\dot H_x^{1/2}} \lsm 1,
}
and
\EQn{\label{eq:localbound-w-onehalf2}
	\norm{ w}_{L_{t,x}^5([0,3])} \lsm \de.
}
%\begin{proof}
%Let $r_0\goe2$ satisfying $\frac{2}{q}+\frac{3}{r_0}=\half 3$. By the Bernstein inequality, Strichartz estimate \eqref{eq:strichartz-2} in Lemma \ref{lem:strichartz}, Littlewood-Paley theory, and fractional chain rule, we have
%\EQn{
%&\sum_{N\in2^\Z}N^{2\ga}\normb{ P_N\int_0^te^{i(t-s)\De}|u(s)|^2u(s)\ds}_{L_t^q L_x^r(I)([0,3])}^2 \\
%\lsm & \sum_{N\in2^\Z}\normb{ P_N |\nabla|^{\half 1}\int_0^te^{i(t-s)\De}|u(s)|^2u(s)\ds}_{L_t^q L_x^{r_0}(I)([0,3])}^2 \\
%\lsm & \sum_{N\in2^\Z}\normb{|\nabla|^{\half 1} P_N\brkb{|u|^2u}}_{L_t^{\frac32} L_x^{\frac{18}{13}}(I)([0,3])}^2 \lsm \normb{|\nabla|^{\half 1} \brkb{|u|^2u}}_{L_t^{\frac32} L_x^{\frac{18}{13}}(I)([0,3])}.
%}
%\end{proof}
Moreover, by Lemma \ref{lem:strichartz} and Lemma \ref{lem:bilinearstrichartz}, we have the following $l^1$ nonlinear estimate (see \cite{Dod20NLS} for its proof),
\EQn{\label{eq:l1nonlinearestimate}
	\sum_{N} N^\frac12\norm{P_N \brk{|v|^2 v}}_{L_t^1 L_x^2(\R)} \lsm t_0^{30}\de^3.
}
We denote
\EQn{\label{eq:an}
	A(N):=\norm{P_Nv_0}_{L_x^2} + \norm{P_N\brk{|v|^2v}}_{L_t^{1}L_x^{2}(\R)},
}
and
\EQn{\label{eq:bn}
	B(N):=\norm{P_Nw_0}_{L_x^2} + \norm{P_N\brk{|w|^2w}}_{L_t^{1}L_x^{2}([0,3])}.
}
Similarly, we denote $\wt{A}(N)$ and $\wt{B}(N)$ with projector $P_N$ replaced by $P_{\sim N}$. Then, by \eqref{eq:localestimate-u}, \eqref{eq:globalbound-v-onehalfregularity} and \eqref{eq:localbound-w-onehalf}, we have  $l_N^2$-bound:
\EQn{\label{eq:l2bound-full}
\brkb{\sum_N N\brkb{A(N)^2+\wt{A}(N)^2}}^{\half 1} \lsm t_0^{10}\de, \text{ and } \brkb{\sum_N N\brkb{B(N)^2+\wt{B}(N)^2}}^{\half 1} \lsm 1.
}

The first main result in this section is that on local time interval $[0,3]$, $v$ belongs to some subcritical space away from the origin, which is inspired by Proposition 4.4 in \cite{BDSW}.
\begin{lem}\label{lem:local-v}
	Suppose that $\half 1 < s \loe 1$ and $\norm{v_0}_{\dot{H}_x^{1/2}\cap \dot{W}_x^{s,1}}\loe t_0^{10}\de$. Let $v$ be the solution of \eqref{eq:small}. Then, for any $0<\ep < \frac{1}{10}\brkb{ s-\half 1}$, 
	\EQn{\label{eq:localbound-v-supercritical-onefrequency}
		\norm{v}_{X\brk{[0,3]}}:=\sup_{N\goe 1} \normb{t^{\frac{9}{10}}N^{\half 1 +\ep}v_N}_{L_{t,x}^5([0,3])} \lsm t_0^{10}\de.
	}
\end{lem}
\begin{proof}
	We bound the linear part using the dispersive estimate:
	\EQ{
		\normb{t^{\frac{9}{10}}N^{\half 1 +\ep}P_N\vl}_{L_{t,x}^5([0,3])} \lsm \normb{\aabs{\nabla}^{\half1 + \ep} v_0}_{L_x^{\frac{5}{4}}}.
	}
Note that $v_0=P_{\goe 1}v_0$, then by interpolation, $\frac12+\frac{10}{3}\ep<s$, and $\dot W^{s,1}\hookrightarrow\dot B^s_{1,\I}$,
\EQ{
\normb{\aabs{\nabla}^{\half1 + \ep}v_0}_{L_x^{\frac{5}{4}}} \lsm & \sum_{N\goe 1} N^{\frac12+\ep}\norm{P_Nv_0}_{L_x^{\frac54}} \\
\lsm & 	\sum_{N\goe 1} N^{-\ep} \brkb{N^{\frac12+\frac{10}{3}\ep}\norm{P_Nv_0}_{L_x^{1}}}^{\frac35} \brkb{N^{\frac12}\norm{P_Nv_0}_{L_x^{2}}}^{\frac25} \\
\lsm & \norm{v_0}_{\dot W_x^{s,1}}^{\frac35} \norm{v_0}_{\dot H_x^{1/2}}^{\frac25} \\
\lsm & t_0^{10} \de.
}
Then, we have
	\EQn{\label{esti:localsuper-linear}
		\normb{t^{\frac{9}{10}}N^{\half 1 +\ep}P_N \vl}_{L_{t,x}^5([0,3])} \lsm t_0^{10} \de.
	}
	
	For the non-linear part, $\normo{t^{9/10}N^{1/2 +\ep}P_N \vnl}_{L_{t,x}^5([0,3])}$ can be bounded by
	\EQnn{
		& \normb{t^{\frac{9}{10}}N^{\half 1 +\ep} \int_0^{t/2} e^{i(t-s)\De}\chi(s)P_N\brk{|v|^2v}\ds}_{L_{t,x}^5\brk{[0,3]}} \label{eq:localsuper-nonlinear-smalltime}\\
		& +\normb{t^{\frac{9}{10}}N^{\half 1 +\ep} \int_{t/2}^t e^{i(t-s)\De}\chi(s)P_N\brk{|v|^2v}\ds}_{L_{t,x}^5\brk{[0,3]}} \label{eq:localsuper-nonlinear-largetime}.
	}
	We bound the term \eqref{eq:localsuper-nonlinear-smalltime} by the dispersive estimate, noting that $t\sim|t-s|$ when $s\loe t/2$,
	\EQn{\label{esti:localsuper-nonlinear-smalltime}
		\eqref{eq:localsuper-nonlinear-smalltime} \lsm & \normb{N^{\half1+\ep}\int_0^{2} \norm{P_N\brk{|v|^2v}}_{L_x^{5/4}}\ds}_{L_t^5([0,3])} \\
		\lsm &  \normb{N^{\half1+\ep}P_N\brk{|v|^2v}}_{L_t^1L_x^{5/4}([0,2])} \\
		\lsm & \normb{N^{\half1+\ep}P_N \brk{v_{\gsm N}^2(v + v_{\ll N}) + v_{\sim N}v_{\ll N}^2}}_{L_t^1L_x^{5/4}([0,2])},
	}
	where we have  used the frequency support property
	\EQ{
		P_N\brk{|v|^2v} = P_NO\brk{v_{\gsm N}^3 +v_{\gsm N}^2v_{\ll N} + v_{\gsm N}v_{\ll N}^2} = P_NO\brk{v_{\gsm N}^2(v + v_{\ll N}) + v_{\sim N}v_{\ll N}^2}.
	}
	Then by $N\goe 1$, H\"older's inequality and \eqref{eq:globalbound-v-onehalfregularity}, 
	\EQn{\label{esti:localsuper-nonlinear-smalltime-highfrequency}
		&\normb{N^{\half1+\ep}P_N \brk{v_{\gsm N}^2(v + v_{\ll N})}}_{L_t^1L_x^{5/4}([0,2])} \\ 
		\lsm &  N^{\half 1+\ep} \norm{v_{\gtrsim N}}_{L_t^\I L_x^2([0,2])}  \norm{v_{\gtrsim N}}_{L_t^2L_x^{20/3}([0,2])} \norm{v}_{L_t^2L_x^{20/3}([0,2])} \\
		\lsm & N^{\ep-\frac{9}{20}} \norm{v_{\gtrsim N}}_{L_t^\I \dot{H}_x^{1/2}(\R)} \norm{v_{\gtrsim N}}_{L_t^2\dot{W}_x^{9/20,20/3}(\R)} \norm{v}_{L_t^{40/11}L_x^{20/3}([0,2])} \\
		\lsm & \norm{v}_{L_t^\I \dot{H}_x^{1/2}(\R)}\norm{v}_{L_t^2\dot{W}_x^{9/20,20/3}(\R)}\norm{v}_{L_t^{40/11}L_x^{20/3}(\R)}\lsm t_0^{30}\de^3.
	}
	Note that for $N\goe 1$, by \eqref{eq:initialbound-v}, we have
	\EQn{
		N^\ep\wt{A}(N) \lsm \norm{v_0}_{\dot{H}_x^{1/2}} + \normb{\abs{\nabla}^{\half 1}\brk{|v|^2v}}_{L_t^1L_x^2(\R)} \lsm t_0^{10}\de,
	}
	then combining Lemma \ref{lem:bilinearstrichartz}, Lemma \ref{lem:schurtest},  \eqref{eq:globalbound-v-onehalfregularity} and \eqref{eq:l2bound-full}, we have
	\EQn{\label{esti:localsuper-nonlinear-smalltime-lowfrequency}
		& \normb{N^{\half1+\ep}P_NO\brk{ v_{\sim N}v_{\ll N}^2}}_{L_t^1L_x^{5/4}([0,2])} \\
		\lsm & N^{\half 1 +\ep} \sum_{N_1,N_2:N_1\loe N_2 \ll N} \norm{O\brk{v_{\sim N} v_{N_1}}}_{L_{t,x}^2([0,2])} \norm{v_{N_2}}_{L_t^2L_x^{10/3}([0,2])} \\
		\lsm & \sum_{N_1,N_2:N_1\loe N_2 \ll N} N^{\ep} N_1 \wt{A}(N) A(N_1) \norm{v_{N_2}}_{L_{t,x}^{10/3}([0,2])} \\
		\lsm & \de \sum_{N_1,N_2:N_1\loe N_2 \ll N} \frac{N_1^{1/2}}{N_2^{1/2}} N_1^{1/2} A(N_1) \norm{|\nabla|^{1/2}v_{N_2}}_{L_{t,x}^{10/3}(\R)} 
		\lsm t_0^{30} \de^3.
	}

	Next, we estimate \eqref{eq:localsuper-nonlinear-largetime}. Note that we have $P_N\brk{|v|^2v}=P_NO\brk{v_{\gsm N}v^2}$. Using dyadic decomposition in time, Lemma \ref{lem:inhomogeneousstrichartz} and $l^3\subset l^5$, we have
	\EQn{\label{esti:localsuper-nonlinear-largetime}
		\eqref{eq:localsuper-nonlinear-largetime} \lsm & \brkb{ \sum_{M\loe 1, M\in 2^\Z} \normb{M^{\frac{9}{10}}N^{\half 1 +\ep} \int_{t/2}^t e^{i(t-s)\De}\chi(s)P_NO\brk{v_{\gsm N}v^2}\ds}_{L_{t,x}^5\brk{[3M/2,3M]}}^5}^{\rev 5} \\
		\lsm & \brkb{ \sum_{M\loe 1, M\in 2^\Z} \normb{M^{\frac{9}{10}}N^{\half 1 +\ep}  \ph_{3M/4\loe\cdot\loe 3M}(t)P_NO\brk{v_{\gsm N}v^2}}_{L_t^{3} L_x^{\frac{9}{7}}\brk{[0,3]}}^5}^{\rev5} \\
		\lsm &  \normb{t^{\frac{9}{10}}N^{\half 1 +\ep}  P_NO\brk{v_{\gsm N}v^2}}_{L_t^{3} L_x^{\frac{9}{7}}\brk{[0,3]}} \\
		\lsm & N^{\half 1 +\ep} \sum_{N_1:N_1\gsm N} \normb{t^{\frac{9}{10}}v_{N_1}}_{L_{t,x}^5([0,3])} \norm{v}_{L_t^{15}L_x^{\frac{45}{13}}(\R)}^2 \\
		\lsm & t_0^{20}\de^2 \sum_{N_1:N_1\gsm N} \frac{N^{\half 1 +\ep}}{N_1^{\half 1 +\ep}} \normb{t^{\frac{9}{10}}N_1^{\half 1 +\ep} v_{N_1}}_{L_{t,x}^5([0,3])} \lsm t_0^{20}\de^2 \norm{v}_{X([0,3])}.
	}
	
	Combining the estimates \eqref{esti:localsuper-linear}, \eqref{esti:localsuper-nonlinear-smalltime-highfrequency},  \eqref{esti:localsuper-nonlinear-smalltime-lowfrequency}, and \eqref{esti:localsuper-nonlinear-largetime}, we have
	\EQ{
		\norm{v}_{X([0,3])} \lsm t_0^{10}\de + t_0^{20} \de^2 \norm{v}_{X([0,3])}.
	}
	Noting that $t_0<1$ and $\de$ is sufficiently small, this implies \eqref{eq:localbound-v-supercritical-onefrequency}.
\end{proof}
An immediate consequence of Lemma \ref{lem:local-v} is 
\begin{cor}\label{cor:localbound-v-supercritical}
	Suppose that $\half 1<s\loe 1$ and $\norm{v_0}_{\dot{H}_x^{1/2}\cap \dot{W}_x^{s,1}}\loe t_0^{10} \de$, then 
	\EQn{\label{eq:localbound-v-supercritical}
		\sum_{N\in2^\Z} \norm{\abs{\nabla}^{1/2}v_N}_{L_{t,x}^5([1,3])} \lsm t_0^{10}\de.
	}
\end{cor}
Next, we derive the global bound for $v$ when $t\goe 3$, for which the estimate is even better. This lemma is one of the key ingredients in our argument, which gives subcritical estimates for long time and will be frequently used in the following.
\begin{lem}\label{lem:globalbound-v}
	Suppose that $\half 1<s\loe 1$ and $\norm{v_0}_{\dot{H}_x^{1/2}\cap \dot{W}_x^{s,1}}\loe t_0^{10} \de$, and let $v$ be the solution of \eqref{eq:small}. Then, for any $t\goe 3$, we have 
	\EQn{
		\norm{v(t)}_{\dot B^s_{\I,\I}} \lsm & t^{-\frac{3}{2}} t_0^{10}\de, \label{eq:globalbound-v-nablainfty}
	}
	and
	\EQn{\label{eq:globalbound-v-infty}
		\norm{v(t)}_{L_x^\I} \lsm & t^{-\frac{3}{2}} t_0^{10} \de, \text{ and}  \quad \norm{v(t)}_{L_x^4} \lsm  t^{-\frac{3}{8}} t_0^{10}\de.
	}
\end{lem}
\begin{proof}
	First, we prove \eqref{eq:globalbound-v-nablainfty}. 
	For any $t>0$, by the dispersive estimate
	\EQn{\label{esti:dispersive-vlinear-nablainfty}
		\norm{\vl}_{\dot B^s_{\I,\I}}  \lsm t^{-3/2}\norm{v_0}_{\dot B^s_{1,\I}} \lsm t^{-3/2}\norm{v_0}_{\dot W_x^{s,1}} \lsm t^{-3/2}t_0^{10}\de.
	}
	For the non-linear part, we have
	\EQn{\label{esti:dispersive-vnonlinear}
		\norm{\vnl}_{\dot B^1_{\I,\I}} \lsm & \int_0^2 |t-s|^{-3/2}\sup_{N\in2^\Z}\norm{NP_N\brk{|v|^2 v}}_{L_x^1} \ds \\
		\lsm & \int_0^2 |t-s|^{-3/2}\sup_{N\in2^\Z}\norm{P_N\nabla\brk{|v|^2 v}}_{L_x^1} \ds \\
		\lsm & t^{-3/2} \norm{O\brk{\nabla v v^2}}_{L_{t,x}^1([0,2])} \\
		\lsm & t^{-3/2}\normb{\sum_N O\brk{\nabla v_N v_{\gsm N} v}+O\brk{\nabla v_N v_{\ll N} v_{\ll N}}}_{L_{t,x}^1([0,2])}.
	}
	By H\"older's inequality, Lemma \ref{lem:schurtest}, and \eqref{eq:globalbound-v-onehalfregularity},
	\EQn{\label{esti:dispersive-vnonlinear-relativelyhigh}
		\normb{\sum_NO\brk{\nabla v_N v_{\gsm N} v}}_{L_{t,x}^1([0,2])} \lsm  \sum_{N_1\gsm N} \norm{O\brk{\nabla v_N v_{N_1} v}}_{L_{t,x}^1([0,2])} &\\
		\lsm  \sum_{N_1\gsm N} \frac{N^{1/2}}{N_1^{1/2}} N^{1/2}\norm{v_N}_{L_{t,x}^2([0,2])} N_1^{1/2}\norm{v_{N_1}}_{L_{t,x}^{\frac{10}{3}}(\R)} \norm{v}_{L_{t,x}^5(\R)} 
		\lsm t_0^{30} \de^3.&
	}
Changing order of summation and using H\"older's inequality,
	\EQn{
	&\normb{\sum_NO\brk{\nabla v_N v_{\ll N} v_{\ll N}}}_{L_{t,x}^1([0,2])} \\
	\lsm & \normb{\sum_NO\brkb{\nabla v_N \sum_{N_1,N_2:N_1\loe N_2\ll N}v_{N_1} v_{N_2}}}_{L_{t,x}^1([0,2])} \\
	\sim & \normb{\sum_{N_2}\sum_{N: N_2\ll N}O\brkb{\nabla v_N \sum_{N_1:N_1\loe N_2}v_{N_1} v_{N_2}}}_{L_{t,x}^1([0,2])} \\
	\lsm & \sum_{N_2} \normb{\sum_{N: N_2\ll N}O\brkb{\nabla v_N \sum_{N_1:N_1\loe N_2}v_{N_1}}}_{L_{t,x}^2([0,2])} \norm{v_{N_2}}_{L_{t,x}^2([0,2])}. \\
	} 
	By the frequency support property, we can update the $l_N^1$-summation for $N$ to $l_N^2$:
	\EQn{\label{esti:l1tol2}
	\normb{\sum_{N: N_2\ll N}O\brkb{\nabla v_N \sum_{N_1:N_1\loe N_2}v_{N_1}}}_{L_{t,x}^2([0,2])} 
	\lsm &  \normb{O\brkb{\nabla v_N \sum_{N_1:N_1\loe N_2}v_{N_1}}}_{l^2_{N:N\gg N_2}L_{t,x}^2([0,2])} \\
	\lsm & \sum_{N_1:N_1\loe N_2} \normb{O\brk{\nabla v_N v_{N_1}}}_{l^2_{N:N\gg N_2}L_{t,x}^2([0,2])}.
	}
	Note also that by \eqref{eq:globalbound-v-subonehalfregularity},
	\EQn{\label{esti:localbound-v-L2L2}
		\brkb{\sum_{N_2} N_2 \norm{v_{N_2}}_{L_{t,x}^2([0,2])}^2}^{1/2} \lsm \norm{v}_{L_{t}^2\dot{H}_x^{1/2}([0,2])} \lsm t_0^{10}\de.
	}
Then using Lemma \ref{lem:bilinearstrichartz}, Lemma \ref{lem:schurtest}, \eqref{eq:l2bound-full}, and \eqref{esti:localbound-v-L2L2},
\EQn{\label{esti:dispersive-vnonlinear-transverse}
&\normb{\sum_NO\brk{\nabla v_N v_{\ll N} v_{\ll N}}}_{L_{t,x}^1([0,2])} \\
\lsm & \sum_{N_1,N_2:N_1\loe N_2} \normb{O\brk{\nabla v_N v_{N_1}}}_{l^2_{N:N\gg N_2}L_{t,x}^2([0,2])} \norm{v_{N_2}}_{L_{t,x}^2([0,2])} \\
\lsm &\sum_{N_1,N_2:N_1\loe N_2}  \brkb{\sum_{N:N\gg N_2}NN_1^2A(N)^2A(N_1)^2}^{1/2}   \norm{v_{N_2}}_{L_{t,x}^2([0,2])} \\
\lsm & t_0^{20} \de \sum_{N_1,N_2:N_1\loe N_2} \frac{N_1^{1/2}}{N_2^{1/2}} N_1^{1/2}A(N_1)N_2^{1/2} \norm{v_{N_2}}_{L_{t,x}^2([0,2])} 
\lsm t_0^{30} \de^3.
}
Therefore, by \eqref{esti:dispersive-vnonlinear},  \eqref{esti:dispersive-vnonlinear-relativelyhigh}, and \eqref{esti:dispersive-vnonlinear-transverse}, we obtain $\norm{\vnl}_{\dot B^1_{\I,\I}}\lsm t^{-\half 3}t_0^{30}\de^3$. By dispersive estimate and $\dot H_x^{\half 1}\subset L_x^3$, $\norm{\vnl}_{\dot B^0_{\I,\I}}\lsm \norm{\vnl}_{L_x^\I} \lsm  t^{-\frac{3}{2}}\int_0^2 \norm{v(s)}_{L_x^3}^3 \ds \lsm t^{-\half 3}t_0^{30}\de^3$. Then, by interpolation, $\norm{\vnl}_{\dot B^{s}_{\I,\I}} \lsm t^{-\half 3}t_0^{30}\de^3$, which completes the proof of \eqref{eq:globalbound-v-nablainfty}.
	
By \eqref{esti:dispersive-vlinear-nablainfty}, we have
\EQn{
\norm{v(t)}_{L_x^\I} \lsm & \sum_{N\goe 1} \norm{P_{N}\vl}_{L_x^\I} + \norm{\vnl}_{L_x^\I} \\
\lsm & t^{-\frac{3}{2}}\sum_{N\goe 1} \norm{P_Nv_0}_{L_x^1} + t^{-\frac{3}{2}}t_0^{30}\de^3 \\
\lsm & t^{-\frac{3}{2}} \norm{v_0}_{\dot W_x^{s,1}} + t^{-\frac{3}{2}}t_0^{30}\de^3 \\
\lsm & t^{-\frac{3}{2}}t_0^{10}\de.
}
Moreover, $L^4$-estimate follows from \eqref{eq:globalbound-v-infty} and the interpolation
	\EQn{
		\norm{v(t)}_{L_x^4} \lsm \norm{v(t)}_{L_x^\I}^{\frac{1}{4}} \norm{v(t)}_{L_x^3}^{\frac{3}{4}}.
	}
This finishes the proof of this lemma.
\end{proof}

\section{Energy estimate}
First, we have the following local estimate of $w$. 	Note that $w$ satisfies
\EQn{
	w(t) = e^{it\De} w_0 - i\int_0^t e^{i(t-s)\De}\brk{|u|^2u-\chi(s)|v|^2v}\ds.
}

\begin{prop}\label{prop:local-w}
	Suppose that $\half 1<s\loe 1$, $\wt u_0\in \dot{H}_x^{1/2}\cap \dot{W}_x^{s,1}$, $u_0(x)=t_0^{1/2}\wt u_0(t_0^{1/2}x)$, and let $v$ and $w$ be the solutions of equations \eqref{eq:small} and \eqref{eq:energy}, respectively. Then,
	\EQn{\label{eq:localbound-w}
		\sup_{0\loe t\loe 3} \norm{w(t)}_{\dot{H}^1} \lsm N_0^{1/2}\norm{u_0}_{\dot{H}_x^{1/2}}.
	}
\end{prop}
\begin{proof}
We define $\wt X([0,3])$ norm as
\EQn{
	\norm{w}_{\wt X([0,3])}:=&\norm{w}_{L_t^\I \dot{H}_x^1([0,3])}+\brkb{ \sum_{N\in2^\Z} N^2 \norm{P_Nw}_{L_{t,x}^{\frac{10}{3}}([0,3])}^2 + N \norm{P_Nw}_{L_{t,x}^{5}([0,3])}^2}^{1/2} \\
	&+ \brkb{ \sum_{N\in2^\Z} N \norm{P_Nw}_{L_t^{2}L_x^\I([0,3])}^2 + N^{4\ep-1} \norm{P_Nw}_{L_t^{1/\ep}L_x^\I([0,3])}^2}^{1/2}.
}
Here $\varepsilon\ll 1$ is a fixed positive small number. 
By \eqref{eq:strichartz-1}  and \eqref{eq:strichartz-3} in Lemma \ref{lem:strichartz} and \eqref{eq:initialbound-w}, we have
\EQn{\label{esti:localbound-w-linear}
\norm{e^{it\De}w_0}_{\wt X([0,3])} \lsm N_0^{1/2}\norm{u_0}_{\dot{H}_x^{1/2}}.
}

For the non-linear part, by the Sobolev inequality and \eqref{eq:strichartz-2} in Lemma \ref{lem:strichartz}, we have
\EQnn{
&\normb{\int_0^t e^{i(t-s)\De}\brk{|u|^2u-\chi(s)|v|^2v}\ds}_{\wt X([0,3])} \nonumber\\ \lsm& \brkb{\sum_N N^2 \norm{P_N\brk{|u|^2u-\chi(s)|v|^2v}}_{L_t^1L_x^2+L_{t,x}^{\frac{10}{7}}([0,3])}^2}^{1/2} \nonumber\\
\lsm & \brkb{\sum_N N^2 \norm{\brk{1-\chi(t)}P_N \brk{|v|^2v}}_{L_{t,x}^{\frac{10}{7}}([0,3])}^2}^{1/2}\label{eq:local-w-awayorigin}\\
&+ \brkb{\sum_N N^2 \norm{P_N\brk{|u|^2u-|v|^2v}}_{L_t^1L_x^2+L_{t,x}^{\frac{10}{7}}([0,3])}^2}^{1/2}.\label{eq:local-w-difference}
}

For the first term \eqref{eq:local-w-awayorigin}, by Minkowski's inequality and the Littlewood-Paley theory, it suffices to estimate
\EQn{
	\norm{\nabla\brk{|v|^2 v}}_{L_{t,x}^{10/7}([1,3])}.
}
To this end, we use the similar method in the proof of  \eqref{eq:globalbound-v-nablainfty} in Lemma \ref{lem:globalbound-v} and  write 
\EQ{
\nabla\brk{|v|^2v} =\sum_N O\brk{\nabla v_N v_{\gsm N} v}+O\brk{\nabla v_N v_{\ll N} v_{\ll N}}.
}
	By H\"older's inequality, Lemma \ref{lem:schurtest}, \eqref{eq:globalbound-v-onehalfregularity}, and \eqref{eq:localbound-v-supercritical},
\EQn{\label{esti:localbound-w-threev-relativelyhigh}
	\normb{\sum_NO\brk{\nabla v_N v_{\gsm N} v}}_{L_{t,x}^{10/7}([1,3])} \lsm  \sum_{N_1\gsm N} \norm{O\brk{\nabla v_N v_{N_1} v}}_{L_{t,x}^{10/7}([1,3])} &\\
	\lsm  \sum_{N_1\gsm N} \frac{N^{1/2}}{N_1^{1/2}} N^{1/2}\norm{v_N}_{L_{t,x}^5([1,3])} N_1^{1/2}\norm{v_{N_1}}_{L_{t,x}^{\frac{10}{3}}(\R)} \norm{v}_{L_{t,x}^5(\R)} 
	\lsm t_0^{30} \de^3.&
}
Using the similar argument in \eqref{esti:l1tol2}, and then using Lemma \ref{lem:bilinearstrichartz}, Lemma \ref{lem:schurtest}, \eqref{eq:l2bound-full}, and \eqref{eq:localbound-v-supercritical}, we have
\EQn{\label{esti:localbound-w-threev-transverse}
	\normb{\sum_NO\brk{\nabla v_N v_{\ll N} v_{\ll N}}}_{L_{t,x}^{10/7}([1,3])}
	\lsm & \sum_{N_1,N_2:N_1\loe N_2} \normb{O\brk{\nabla v_N v_{N_1}}}_{l^2_{N:N\gg N_2}L_{t,x}^2([0,3])} \\
	&\cdot\norm{v_{N_2}}_{L_{t,x}^5([1,3])} \lsm t_0^{30} \de^3.
}
Combining  \eqref{esti:localbound-w-threev-relativelyhigh} and \eqref{esti:localbound-w-threev-transverse}, we have
\EQn{\label{esti:localbound-w-threev-nonlinear}
	\norm{\nabla \big(|v|^2v\big)}_{L_{t,x}^{10/7}([1,3])} \lsm t_0^{30} \de^3.
}

For the term \eqref{eq:local-w-difference}, we have
\EQnn{
	&\brkb{\sum_N N^2 \norm{P_N\brk{|u|^2u-|v|^2v}}_{L_t^1L_x^2+L_{t,x}^{\frac{10}{7}}([0,3])}^2}^{1/2}\nonumber\\
	\lsm & \brkb{\sum_N  N^2 \norm{P_NO\brk{ w_{\gsm N}\brk{w^2+wv+v^2}}}_{L_{t,x}^{\frac{10}{7}}([0,3])}^2}^{1/2} \label{eq:local-w-difference-nablaw}\\
	& + \brkb{\sum_N  N^2 \norm{P_NO\brk{v_{\gsm N}\brk{w^2+wv}}}_{L_t^1L_x^2+L_{t,x}^{\frac{10}{7}}([0,3])}^2}^{1/2}.\label{eq:local-w-difference-nablav}
}
We first estimate the term \eqref{eq:local-w-difference-nablaw}. We only consider the term $O\brk{w_{\gsm N}w^2}$, since other terms can be treated similarly. Since 
\EQ{
P_N O\brk{w_{\gsm N}w^2} = P_N O\brk{w_{\gsm N}w_{\gsm N}w} + P_N O\brk{w_{\sim N}w_{\ll N}^2},
}
by H\"older's inequality, \eqref{eq:localbound-w-onehalf}, \eqref{eq:localbound-w-onehalf2}, and Lemma \ref{lem:schurtest}, we have
\EQn{\label{esti:localbound-w-gradientw-twohigh}
&\brkb{\sum_N  N^2 \norm{P_N O\brk{w_{\gsm N}w_{\gsm N}w}}_{L_{t,x}^{\frac{10}{7}}([0,3])}^2}^{1/2}\\
\lsm& \brkb{\sum_N  \brkb{\sum_{N_1,N_2:N_1,N_2\gsm N}N \norm{w_{N_1}}_{L_{t,x}^5([0,3])} \norm{w_{N_2}}_{L_{t,x}^{\frac{10}{3}}([0,3])} \norm{w}_{L_{t,x}^5([0,3])}}^2}^{1/2} \\
\lsm& \de\brkb{\sum_N  \brkb{\sum_{N_1,N_2:N_1,N_2\gsm N}\frac{N}{N_1^{1/2}N_2^{1/2}} N_1^{1/2} \norm{w_{N_1}}_{L_{t,x}^5([0,3])} N_2^{1/2} \norm{w_{N_2}}_{L_{t,x}^{\frac{10}{3}}([0,3])} }^2}^{1/2}\\
\lsm& \de \sum_{N_1,N_2}\frac{\min\fbrk{N_1,N_2}^{1/2}}{\max\fbrk{N_1,N_2}^{1/2}} N_1^{1/2} \norm{w_{N_1}}_{L_{t,x}^5([0,3])} N_2^{1/2} \norm{w_{N_2}}_{L_{t,x}^{\frac{10}{3}}([0,3])} 
\lsm \de \norm{w}_{\wt X([0,3])}.
}
Similarly using only H\"older's inequality and \eqref{eq:localbound-w-onehalf2}, we also have
\EQn{\label{esti:localbound-w-gradientw-onehigh}
\brkb{\sum_N  N^2 \norm{P_N O\brk{w_{\sim N}w_{\ll N}^2}}_{L_{t,x}^{\frac{10}{7}}([0,3])}^2}^{1/2} \lsm & \brkb{\sum_N  N^2 \norm{w_{\sim N}}_{L_{t,x}^{\frac{10}{3}}([0,3])}^2 \norm{w}_{L_{t,x}^5([0,3])}^4 }^{1/2} \\
\lsm & \de^2 \norm{w}_{\wt X([0,3])}.
}
Therefore, by \eqref{esti:localbound-w-gradientw-twohigh} and \eqref{esti:localbound-w-gradientw-onehigh}, we get 
\EQn{\label{esti:localbound-w-gradientw}
	\eqref{eq:local-w-difference-nablaw}\lsm \de  \norm{w}_{\wt X([0,3])}.
}

For the second term  \eqref{eq:local-w-difference-nablav}, we write 
\EQ{
P_NO\brk{ v_{\gsm N}\brk{w^2+wv}}= &P_NO\brk{ v_{\sim N}\brk{w_{\ll N}^2+w_{\ll N}v_{\ll N}}}\\
&+
P_NO\brk{ v_{\gsm N}w_{\gtrsim N}w+ v_{\gsm N}w_{\gtrsim N}v+v_{\gsm N}v_{\gtrsim N} w}.
}
Treated similarly as \eqref{esti:localbound-w-gradientw-twohigh}, we have 
\EQn{\label{esti:localbound-w-gradientv-relativelyhigh}
	\brkb{\sum_N  N^2\big\|P_NO\brk{ v_{\gsm N}w_{\gtrsim N}w+ v_{\gsm N}w_{\gtrsim N}v+v_{\gsm N}v_{\gtrsim N}w}&\big\|_{L_{t,x}^{\frac{10}{7}}([0,3])}^2}^{1/2} \lsm  
	\de \norm{w}_{\wt X([0,3])}.
}
Therefore, it suffices to estimate
\EQn{\label{eq:local-w-mainterm}
	& \brkb{\sum_N  N^2 \norm{O\brk{ v_{\sim N}\brk{w_{\ll N}^2+w_{\ll N}v_{\ll N}}}}_{L_t^1 L_x^2([0,3])}^2}^{1/2}.
}
By Minkowski's inequality,
\EQn{\label{7.03-0723}
	\eqref{eq:local-w-mainterm} \lsm  \sum_{N_1,N_2} \brkb{\sum_{N:N\gg N_1,N_2}  N^2 \norm{O\brk{v_{\sim N}\brk{v_{N_2}+w_{N_2}}w_{N_1}}}_{L_t^1 L_x^2([0,3])}^2}^{1/2}.
}
For $N_1,N_2\ll N$, $1< q_0 \loe 2$, by Lemma \ref{lem:bilinearstrichartz}, we have the following tri-linear estimate:
\EQn{\label{7.04-0723}
&N \norm{O\brk{v_{\sim N}\brk{v_{N_2}+w_{N_2}}w_{N_1}}}_{L_t^1L_x^2([0,3])}\\ \lsm & N \norm{O\brk{v_{\sim N}\brk{v_{N_2}+w_{N_2}}}}_{L_t^{q_0}L_x^2([0,3])} \norm{w_{N_1}}_{L_t^{q_0'}L_x^\I([0,3])} \\
\lsm & N^{\half 1} N_2^{2-\frac{2}{q_0}}\wt{A}(N) (A(N_2)+B(N_2)) N_1^{\frac{2}{q_0}-\half 3} N_1^{\frac{2}{q_0'}-\half 1} \norm{w_{N_1}}_{L_t^{q_0'}L_x^\I([0,3])} \\
\lsm & N^{\half 1} \wt{A}(N) \brkb{\frac{N_2}{N_1}}^{\half 3-\frac{2}{q_0}} N_2^{\half 1}(A(N_2)+B(N_2))N_1^{\frac{2}{q_0'}-\half 1} \norm{w_{N_1}}_{L_t^{q_0'}L_x^\I([0,3])}.
}
Therefore, by \eqref{7.03-0723}, \eqref{7.04-0723} and \eqref{eq:l2bound-full},
we have
\EQn{
\eqref{eq:local-w-mainterm} \lsm t_0^{10}\de \sum_{N_1,N_2} \brkb{\frac{N_2}{N_1}}^{\half 3-\frac{2}{q_0}} N_2^{\half 1}(A(N_2)+B(N_2))N_1^{\frac{2}{q_0'}-\half 1} \norm{w_{N_1}}_{L_t^{q_0'}L_x^\I([0,3])}.
}
Note that for both choices $q_0=2$ and $q_0=1/(1-\ep)$ for $\varepsilon\ll 1$, we have
\EQn{\label{esti:local-w-inductionhypothesis}
\normb{N_1^{\frac{2}{q_0'}-\half 1} \norm{w_{N_1}}_{L_t^{q_0'}L_x^\I([0,3])}}_{l_{N_1}^2} \lsm \norm{w}_{\wt X([0,3])}.
}
Therefore, we can choose $q_0=1/(1-\ep)$ when $N_1\loe N_2$, and $q_0=2$ when $N_1\goe N_2$. Using Lemma \ref{lem:schurtest},  \eqref{eq:l2bound-full} and \eqref{esti:local-w-inductionhypothesis}, we have
\EQn{\label{esti:localbound-w-gradientv-tranverse}
\eqref{eq:local-w-mainterm} \lsm & t_0^{10}\de \sum_{N_1\loe N_2}
\brkb{\frac{N_1}{N_2}}^{\half 1-2\ep} N_2^{\half 1}(A(N_2)+B(N_2))N_1^{2\ep-\half 1} \norm{w_{N_1}}_{L_t^{1/\ep}L_x^\I([0,3])} \\
&+ t_0^{10}\de\sum_{N_1\goe N_2}
\brkb{\frac{N_2}{N_1}}^{\half 1} N_2^{\half 1}(A(N_2)+B(N_2))N_1^{\half 1} \norm{w_{N_1}}_{L_t^2L_x^\I([0,3])} \\
\lsm &t_0^{10}\de (t_0^{10}\de+1)  \norm{w}_{\wt X([0,3])}.
}

Together with \eqref{esti:localbound-w-linear},  \eqref{esti:localbound-w-threev-nonlinear},  \eqref{esti:localbound-w-gradientw},  \eqref{esti:localbound-w-gradientv-relativelyhigh}, and \eqref{esti:localbound-w-gradientv-tranverse}, we have
\EQn{
	\norm{w}_{\wt X([0,3])}
	\lsm & N_0^{1/2}\norm{u_0}_{\dot{H}_x^{1/2}} + t_0^{30}\de^3 + t_0^{10}\de  \norm{w}_{\wt X([0,3])},
}
then by choosing $\delta$ small enough, we have
\EQ{
	\norm{w}_{\wt X([0,3])}\lsm N_0^{1/2}\norm{u_0}_{\dot{H}_x^{1/2}},
}
which completes the proof of \eqref{eq:localbound-w}.
\end{proof}

Define that
\EQ{
	E(w(t)) = \half 1\int_{\R^3} \abs{\nabla w(t,x)}^2 \dx + \rev 4\int_{\R^3}  \abs{w(t,x)}^4 \dx.
}
By \eqref{eq:localbound-w}, \eqref{eq:initialbound-u}, \eqref{eq:initialbound-v}, Sobolev's inequality, and the interpolation,
\EQ{
	\sup_{0\loe t\loe 3}\norm{w(t)}_{L_x^4} \lsm \sup_{0\loe t\loe 3}\norm{w(t)}_{L_x^3}^{\frac12} \norm{w(t)}_{L_x^6}^{\frac12} \lsm N_0^{1/4}.
}
Then, there exists a constant $C_0=C_0(\norm{\wt u_0}_{\dot H_x^{1/2}\cap\dot W_x^{s,1}})>0$ such that 
\EQn{\label{eq:localbound-w-c0}
	\sup_{0\loe t\loe 3} E(w(t)) \loe C_0 N_0.
}

Next, we turn to the energy bound of $w$ when $t\goe 3$. For all $t\goe 0$, by conservation law $\norm{u(t)}_{L_x^2}^2 = \norm{u_0}_{L_x^2}^2$,
and  \eqref{eq:globalbound-v-subonehalfregularity}, we have
\EQn{\label{eq:globalbound-w-l2}
\norm{w(t)}_{L_x^2} \lsm t_0^{\frac s2-1}.
} 
We take $3<T_0<\I$ such that $u$ is well-posed on $[0,T_0)$. 
We are going to prove that
\begin{prop}\label{prop:globalbound-w-energy}
	Suppose that $12/13 < s\loe 1$, $\wt u_0\in \dot{H}_x^{1/2}\cap \dot{W}_x^{s,1}$, and $u_0(x)=t_0^{1/2}\wt u_0(t_0^{1/2}x)$. Let $v$ and $w$ be the solution of equations \eqref{eq:small} and \eqref{eq:energy}, respectively. Then, 
	\EQn{\label{eq:globalbound-w-energy}
		\sup_{3\loe t<T_0} E(w(t)) <2C_0 N_0.
	}
\end{prop}
\begin{proof}
Suppose that \eqref{eq:globalbound-w-energy} did fail, then we could define
	\EQ{
		T:=\inf\fbrkb{3<T_1< T_0 : \sup_{3\loe t<T_1} E(w(t)) \goe 2C_0 N_0 .}
	}
In the following, we denote $\lsm_{C_0}$ as $\lsm$  for short.
	
	We first collect some useful estimates. By definition of $T$, for $0 \loe t\loe T$, we have that
	\EQn{\label{esti:globalbound-w-inductionhypothesis}
	\norm{w(t)}_{\dot H_x^1}\lsm N_0^{1/2} \text{, and }\norm{w(t)}_{L_x^4}\lsm N_0^{1/4} .
	} 
Combining \eqref{eq:globalbound-w-l2}, for any $0\loe l\loe1$,
\EQn{\label{eq:globalbound-u-onehalf}
\norm{w}_{L_t^\I \dot{H}_x^{l}([0,T])}+\norm{u}_{L_t^\I \dot{H}_x^{l}([0,T])} \lsm N_0^{\frac12 l}t_0^{\frac{1-l}{2}(s-2)}.
}
Using the interaction Morawetz inequality in Lemma \ref{lem:interaction}, \eqref{eq:globalbound-u-onehalf} and \eqref{eq:globalbound-v-subonehalfregularity}, we have
\EQn{\label{eq:globalbound-w-l44}
\norm{v}_{L_{t,x}^4([0,T])}\lsm t_0^{10}\de \text{, and } \norm{w}_{L_{t,x}^4([0,T])} \lsm N_0^{1/8} t_0^{\frac 32(s-2)}.
}
Interpolating with $\norm{w}_{L_t^\I L_x^6} \lsm N_0^{1/2}$,
\EQn{\label{esti:globalbound-w-l6l9/2}
\norm{w}_{L_t^6 L_x^{9/2}([0,T])} \lsm  N_0^{1/4}t_0^{s-2}.
}
Using the equation \eqref{eq:energy} and \eqref{esti:globalbound-w-l6l9/2}, we have 
\EQn{\label{eq:globalbound-w-l2infty}
\norm{w}_{L_t^3L_x^9([0,T])} &\lsm N_0^{1/4} + \normb{\abs{\nabla}^{\half 1}(|u|^2u)}_{L_t^2L_x^{\frac65}([0,T])}\\ &\lsm N_0^{1/4} + \normb{\abs{\nabla}^{\half 1}u}_{L_t^\I L_x^3([0,T])}\norm{u}_{L_{t,x}^4([0,T])}^2 \lsm N_0^{3/4}t_0^{3(s-2)}.
}
Then, by \eqref{esti:globalbound-w-inductionhypothesis}, \eqref{eq:globalbound-w-l2},  \eqref{esti:globalbound-w-l6l9/2}, and \eqref{eq:globalbound-w-l2infty}, for any $0\loe \al \loe 1$ and $N\in2^\Z$, we have
\EQn{\label{eq:globalbound-w-full}
&N^\al\norm{P_N w_0}_{L_x^2} + N^\al \norm{P_N\brk{|u|^2u}}_{L_t^2L_x^{\frac65}([0,T])} \\
\lsm & N_0^{\half \al} + \norm{u}_{L_t^\I\dot H_x^\al([0,T])} \norm{u}_{L_t^6L_x^{\frac92}([0,T])} \norm{u}_{L_t^3L_x^9([0,T])} \lsm  N_0^{\half \al+1} t_0^{(\frac92-\frac12\al)(s-2)}. 
}
Moreover, for $0\loe \al\loe\half 1$, we can improve the bound by \eqref{esti:globalbound-w-inductionhypothesis}, \eqref{eq:globalbound-w-l2}, and \eqref{eq:globalbound-w-l44}, 
\EQn{\label{eq:globalbound-w-full2}
&N^\al\norm{P_N w_0}_{L_x^2} + N^\al \norm{P_N\brk{|u|^2u}}_{L_t^2L_x^{\frac65}([0,T])} \\
\lsm & N_0^{\half \al} + \norm{\abs{\nabla}^\al u}_{L_t^\I L_x^3([0,T])} \norm{u}_{L_{t,x}^4([0,T])}^2 \lsm  N_0^{\half \al+\rev 2}t_0^{(\frac{13}{4}-\frac12\al)(s-2)}.
}

Inspired by \cite{Dod20NLS}, we define the modified energy by
\EQn{
	\E(w(t)) = E(w(t)) + \jb{|w|^2w,v}.
}
Then, we have
\EQn{\label{eq:modifiedenergy}
	\frac{\mathrm{d}}{\mathrm{d}t} \E(w(t)) = & \jb{|w|^2w,v_t} - \jb{w_t,|v|^2v} -\jb{w_t,2|v|^2w+v^2\wb{w}}.
}
Therefore, noting that $C_0N_0\gsm 1$, 
	\EQn{\label{esti:energyerror}
		\sup_{3\loe t\loe T}\abs{\E(w(t))-E(w(t))} \loe \rev{100} C_0 N_0.
	}
		Note that for $t\goe 3$, $v_t=i\De v$. After integrating by parts in $x$, we have
	\EQn{
		\int_3^T\aabs{\jb{|w|^2w,v_t}}\dt \lsm  \int_3^T\absb{\int O\brk{\nabla w\cdot\nabla v w^2}\dx}\dt.
	}
	Therefore, for any $3\loe t\loe T$, by \eqref{eq:localbound-w-c0}, \eqref{esti:energyerror}, and \eqref{eq:modifiedenergy}, it follows that
	\EQn{
		\E(w(t)) \loe &  \E(w(3)) +  \int_3^t\absb{ \frac{\mathrm{d}}{\mathrm{d}s} \E(w(s))} \ds \\
		\loe & C_0N_0  + \rev{100} C_0N_0  + \int_3^T\absb{\jb{|w|^2w,v_t}}\dt + C \int_3^T\int(|w|^4+|v|^4)|v|^2\dx\dt \\
		& + C\int_3^T\int\aabs{\nabla w v}^2\dx\dt +C\int_3^T\absb{\int O\brkb{\nabla w\cdot\nabla v v(w +v)}\dx}\dt \\
		\loe &  \frac{101}{100} C_0N_0 + \int_3^T\int(|w|^4+|v|^4)|v|^2\dx\dt  \\
	&+ C\int_3^T\int\aabs{\nabla w v}^2\dx\dt +C\int_3^T\absb{\int O\brkb{\nabla w\cdot\nabla v (w^2+wv+v^2)}\dx}\dt.
}

By \eqref{eq:globalbound-w-l44} and  \eqref{eq:globalbound-v-infty}, we have
\EQn{\label{esti:energybound-potential}
\int_3^T\int(|w|^4+|v|^4)|v|^2\dx\dt \lsm & \brkb{\norm{w}_{L_{t,x}^4([3,T])}^4+\norm{v}_{L_{t,x}^4([3,T])}^4} \norm{v}_{L_t^\I L_x^\I([3,T])}^2 \\
\lsm &  N_0^{1/2}t_0^{6(s-2)}(t_0^{10}\de)^2 \lsm \de^2 N_0^{1/2}.
}
By \eqref{esti:globalbound-w-inductionhypothesis} and  \eqref{eq:globalbound-v-infty}, we also have
\EQn{\label{esti:energybound-nogradientv}
\int_3^T\int\aabs{\nabla w v}^2\dx\dt \lsm \norm{w}_{L_t^\I\dot{H}_x^1([3,T])}^2\norm{v}_{L_t^2L_x^\I([3,T])}^2 \lsm \de^2 N_0.
}
By \eqref{eq:globalbound-w-l44} and \eqref{eq:globalbound-v-infty},
we have
\EQn{\label{eq:energybound-mainterm-lowfrequency}
	&\int_3^T\absb{\int O\brkb{\nabla w\cdot\nabla v_{\loe 1} (w^2+wv+v^2)}\dx}\dt \\
	&\lsm \norm{\nabla w}_{L_{t}^\I L_x^2([3,T])} \norm{ v}_{L_{t}^2 L_x^\I([3,T])} \brkb{\norm{w}_{L_{t,x}^4([3,T])}^2+\norm{v}_{L_{t,x}^4([3,T])}^2} \\
	&\lsm N_0^{1/2}\cdot t_0^{10}\de\cdot N_0^{1/4}t_0^{3(s-2)} \lsm \de N_0^{3/4}.
}
Then, it suffices to consider the term
\EQn{\label{eq:energybound-mainterm}
\int_3^T\absb{\int O\brkb{\nabla w\cdot\nabla v_{\goe 1} (w^2+wv+v^2)}\dx}\dt.
}

By the frequency support property, we can divide it into three terms,
\EQnn{
\eqref{eq:energybound-mainterm} \lsm & \sum_{N\goe 1} \int_3^T\int\absb{\nabla v_N \cdot\nabla w w_{\gsm N}(v+w)}\dx\dt \label{eq:energybound-mainterm-relativelyhigh-w}\\
& + \sum_{N\goe 1} \int_3^T\int\absb{\nabla v_N \cdot\nabla w v_{\gsm N}(v+w)}\dx\dt \label{eq:energybound-mainterm-relativelyhigh-v}\\
& + \sum_{N\goe 1} \int_3^T\int\absb{\nabla v_N \cdot\nabla w_N \brkb{w_{\ll N}^2 + v_{\ll N}w_{\ll N} + v_{\ll N}^2}}\dx\dt. \label{eq:energybound-mainterm-transverse}
}
For \eqref{eq:energybound-mainterm-relativelyhigh-w}, from \eqref{esti:globalbound-w-inductionhypothesis},  \eqref{eq:globalbound-w-l44}, and the Sobolev inequality, we have for any $1\goe s>\frac56$,
\EQn{\label{esti:energybound-mainterm-relativelyhigh-w-1}
&\int_3^T\int\absb{\nabla v_N \cdot\nabla w w_{\gsm N}(v+w)}\dx\dt\\
\lsm & \norm{\nabla v_N}_{L_t^{\frac{3}{2}}L_x^\I([3,T])} \norm{\nabla w}_{L_t^\I L_x^2([3,T])}\norm{w_{\gsm N}}_{L_{t,x}^4([3,T])}^{\rev 3} \norm{w_{\gsm N}}_{L_t^\I L_x^4([3,T])}^{\frac23}\\
&\cdot \brkb{\norm{v}_{L_{t,x}^4([3,T])} + \norm{w}_{L_{t,x}^4([3,T])}}\\
\lsm & N_0^{\frac{2}{3}} t_0^{2(s-2)} \norm{\nabla v_N}_{L_t^{\frac{3}{2}}L_x^\I([3,T])} \norm{w_{\gsm N}}_{L_t^\I \dot{H}_x^{\frac34}([3,T])}^{\frac23} 
\lsm  N_0 t_0^{\frac{25}{12}(s-2)} \normb{\abs{\nabla}^{\frac56} v_N}_{L_t^{\frac{3}{2}}L_x^\I([3,T])}.
}
Hence, applying  Lemma \ref{lem:globalbound-v}, we have
\EQn{\label{esti:energybound-mainterm-relativelyhigh-w}
\eqref{eq:energybound-mainterm-relativelyhigh-w} \lsm  N_{0} \sum_{N\goe 1} N^{\frac56-s} t_0^{\frac{25}{12}(s-2)} \cdot t_0^{10}\de \lsm \de N_0.
}

For \eqref{eq:energybound-mainterm-relativelyhigh-v}, by H\"older's inequality, \eqref{esti:globalbound-w-inductionhypothesis}, \eqref{eq:globalbound-v-subonehalfregularity}, and \eqref{eq:globalbound-w-l2}, 
\EQn{
&\int_3^T\int\absb{\nabla v_N \cdot\nabla w v_{\gsm N}(v+w)}\dx\dt \\
\lsm & \norm{\nabla v_N v_{\gsm N}}_{L_t^1L_x^\I([3,T])} \norm{\nabla w}_{L_t^\I L_x^2([3,T])}\brkb{\norm{v}_{L_t^\I L_x^2([3,T])}+\norm{w}_{L_t^\I L_x^2([3,T])}} \\
\lsm & N_0^{1/2} N^{1-2s} t_0^{\frac{s-2}{2}} \norm{\abs{\nabla}^{s}v_N}_{L_t^2L_x^\I([3,T])} \norm{\abs{\nabla}^{s}v_{\gsm N}}_{L_t^2L_x^\I([3,T])} . 
}
Therefore, using \eqref{eq:globalbound-v-nablainfty} and \eqref{eq:globalbound-v-infty}, we have for any $1\goe s>\half 1$,
\EQn{\label{esti:energybound-mainterm-relativelyhigh-v}
	\eqref{eq:energybound-mainterm-relativelyhigh-v}\lsm \sum_{N\goe 1} N_0^{1/2} N^{1-2s} t_0^{\frac{s-2}{2}} \cdot (t_0^{10}\de)^2 \lsm \de^2 N_0^{1/2}.
}

Finally, we deal with the term \eqref{eq:energybound-mainterm-transverse}. Since the estimate of $v$ will not increase the bound of $N_0$, the worst case is as follows, and the others can be treated similarly: 
\EQn{\label{eq:energybound-mainterm-transverse-simple}
\sum_{N\goe 1} \int_3^T\int\absb{\nabla v_N \cdot\nabla w_N w_{\ll N} w_{\ll N}}\dx\dt.
}
By H\"older's inequality and \eqref{eq:globalbound-w-l44}, 
\EQn{\label{esti:energybound-mainterm-transverse-simple-1}
&\int_3^T\int\absb{\nabla v_N \cdot\nabla w_N w_{\ll N} w_{\ll N}}\dx\dt \\
\lsm & \norm{\nabla v_N}_{L_t^{\frac{52}{25}+}L_x^\I([3,T])} \norm{\nabla w_N w_{\ll N}}_{L_t^{\frac{8}{5}-}L_x^2([0,T])}^{\frac{2}{13}} \norm{\nabla w}_{L_t^\I L_x^2([3,T])}^{\frac{11}{13}} \norm{w_{\ll N}}_{L_t^{\frac{48}{11}}L_x^{\frac{48}{13}}([3,T])}^{\frac{24}{13}}.
}
Interpolating \eqref{eq:globalbound-w-l2} with \eqref{eq:globalbound-w-l44}, we have
\EQn{\label{esti:energybound-mainterm-transverse-simple-2}
\norm{w_{\ll N}}_{L_t^{\frac{48}{11}}L_x^{\frac{48}{13}}([3,T])} \lsm \norm{w}_{L_t^\I L_x^2([3,T])}^{\rev{12}} \norm{w}_{L_{t,x}^4([3,T])}^{\frac{11}{12}} \lsm N_0^{\frac{1}{8}} t_0^{\frac{17}{12}(s-2)}.
}
By Lemma \ref{lem:bilinearstrichartz},  \eqref{eq:globalbound-w-full}, and \eqref{eq:globalbound-w-full2}, we have
\EQn{\label{esti:energybound-mainterm-transverse-simple-3}
\norm{\nabla w_N w_{\ll N}}_{L_t^{\frac{8}{5}-}L_x^2([0,T])} \lsm & \sum_{N_1:N_1\ll N} N^{-\half 1} N\brkb{\norm{P_N w_0}_{L_x^2} +  \norm{P_N\brk{|u|^2u}}_{L_t^{2}L_x^{\frac65}([0,T])}} \\
& \cdot N_1^{\half 1-}\brkb{\norm{P_{N_1} w_0}_{L_x^2} +  \norm{P_{N_1}\brk{|u|^2u}}_{L_t^2L_x^{\frac65}([0,T])}} \\
\lsm & N^{-\half 1} N_0^{\frac{9}{4}-}t_0^{7(s-2)-}.
}
By \eqref{esti:globalbound-w-inductionhypothesis}, \eqref{esti:energybound-mainterm-transverse-simple-1},  \eqref{esti:energybound-mainterm-transverse-simple-2}, \eqref{esti:energybound-mainterm-transverse-simple-3}, and \eqref{eq:globalbound-v-nablainfty}, we have for any $1\goe s>\frac{12}{13}$,
\EQn{\label{esti:energybound-mainterm-transverse-simple}
	\eqref{eq:energybound-mainterm-transverse-simple} \lsm & N_0^{1-}\sum_{N\goe 1} N^{\frac{12}{13}-s} t_0^{\frac{48}{13}(s-2)-} \normb{\abs{\nabla}^{s} v_N}_{L_t^{\frac{52}{25}+}L_x^\I([3,T])} \\
	\lsm & N_0^{1-}\sum_{N\goe 1} N^{\frac{12}{13}-s} t_0^{\frac{48}{13}(s-2)-} \cdot t_0^{10}\de \lsm \de N_0.
}
Then, by \eqref{esti:energybound-mainterm-relativelyhigh-w}, \eqref{esti:energybound-mainterm-relativelyhigh-v} and \eqref{esti:energybound-mainterm-transverse-simple}, we have
\EQn{\label{esti:energybound-gradientv}
\eqref{eq:energybound-mainterm} \lsm \de N_0.
}

Therefore,  \eqref{esti:energybound-potential}, \eqref{esti:energybound-nogradientv}, \eqref{eq:energybound-mainterm-lowfrequency}, and \eqref{esti:energybound-gradientv} imply
\EQn{
	\sup_{t\in [3,T]}\E(w(t)) \loe \frac{101}{100} C_0N_0 + C(C_0)\de N_0 \loe \half 3 C_0N_0,
}
which contradicts to the definition of $T$. 
Then, we obtain \eqref{eq:globalbound-w-energy}.
\end{proof}

\begin{proof}[Proof of Theorem \ref{thm:main}]
Note that by \eqref{eq:globalbound-v-subonehalfregularity}, \eqref{eq:globalbound-w-l2} and Proposition \ref{prop:globalbound-w-energy}, we obtain that for some $C(N_0, \|\wt u_0\|_{\dot H_x^\frac12\cap \dot W_x^{s,1}})>0$, such that  
$$
\|u\|_{L^\infty_tH^\frac12_x(I)}\le C(N_0, \|\wt u_0\|_{\dot H_x^\frac12\cap \dot W_x^{s,1}}).
$$
Here $I$ is the maximal lifespan.
Then the global well-posedness and scattering follows directly from the conclusion in \cite{KM10TranAMS}. 
This finishes the proof of Theorem \ref{thm:main}.
\end{proof}

\section*{Acknowledgment}
The authors are very grateful to Professors Benjamin Dodson and Timothy Candy for helpful comments and discussions.  
J. Shen was supported by China Scholarship Council
201706010021. 
Y. Wu is partially supported by NSFC 11771325 and 11571118.

\end{document}